\documentclass[a4paper,reqno]{amsart}
\usepackage{amsmath}
\usepackage{amsfonts}
\usepackage{amssymb}

\newtheorem{theorem}{Theorem}[section]
\newtheorem{lemma}[theorem]{Lemma}

\newtheorem{proposition}[theorem]{Proposition}
\newtheorem{corollary}[theorem]{Corollary}
\newtheorem{remark}[theorem]{Remark}

\newtheorem{question}[theorem]{Question}

\newcommand{\RR}{\mathbb{R}}

\newcommand{\FF}{\mathbb{F}}

\newcommand{\V}{\Vert}

\begin{document}

\title[Restriction theorem for general measures]{Finite field analogue of restriction theorem for general measures}

\author{Changhao Chen}

\address{School of Mathematics and Statistics, The University of New South Wales, Sydney NSW 2052, Australia }
\email{changhao.chenm@gmail.com}

%\date{\today}

%\keywords{ }

%\thanks{The author acknowledges the support of the Vilho, Yrj\"o, and Kalle V\"ais\"al\"a foundation.}

\begin{abstract}
We study restriction problem in vector spaces over finite fields. We obtain finite field analogue of  Mockenhaupt-Mitsis-Bak-Seenger  restriction theorem, and we show that the range of the exponentials is sharp. 
\end{abstract}

%%%%%%%%%%%%%%%%%%%%%%%%%%%%%%

\maketitle

\section{Introduction}

Mockenhaupt and Tao \cite{MT} initially studied the finite field restriction problem.  We obtain finite field analogue of  Mockenhaupt-Mitsis-Bak-Seenger  restriction theorem. Furthermore,  we show that the range of the exponentials are sharp for the finite fields setting.

We first introduce the restriction problem in Euclidean spaces. We refer to Mattila \cite[Chapter 19]{Mattila}, Tao \cite{Tao1}, Wolff \cite[Chapter 7]{Wolff} for more details. Let $f:\RR^{n}\rightarrow \mathbb{C}$ be a function and $\mu$ be a Borel measure on $\RR^{n}$. For $\xi\in \RR^{n}$, define 
\[
\widehat{f\mu}(\xi)=\int e^{-2\pi i x\cdot\xi}f(x)d \mu (x).
\]
The restriction problem determinates the values $p, q$ such that 
\begin{equation*}
\Vert\widehat{f \mu}\Vert_{L^{q}(\RR^{n})}\lesssim_{p,q}\Vert f\Vert_{L^p(\mu)}, \forall f\in L^{p}(\mu).
\end{equation*}
 
We write $X \lesssim Y$ if there is a positive constant $C$ such that $X\leq C Y$, $X\gtrsim Y$ if $Y\lesssim X$, and $X\approx Y$ if $X\lesssim Y$ and $X\gtrsim Y$. Furthermore $X\lesssim_{t}Y$ means that the constant depends on $t$ only.

Mockenhaupt \cite{Mockenhaupt}, Mitsis \cite{Mitsis} independently proved the following Theorem \ref{thm:MM} for the range $q>q_{n, \alpha, \beta}=(4n-4\alpha+2\beta)/\beta$. The endpoint case $q=q_{n, \alpha, \beta}$   was proved by Bak and Seeger \cite{BakSeeger}. We note that if $\mu$ is the nature measure on sphere, then $\alpha=\beta=\frac{n-1}{2}$ and Theorem \ref{thm:MM} becomes the classical Stein-Tomas restriction theorem. We refer to \cite[Chapter 19]{Mattila} for more details.

\begin{theorem}\label{thm:MM}
Let $\mu $ be a finite Borel measure on $\RR^{n}$ such that
\begin{equation}\label{eq:A0}
\mu(B(x,r))\lesssim r^{\alpha}, \, \forall x\in \RR^{n}, \, r>0, 
\end{equation}
and 
\begin{equation}\label{eq:B0}
|\widehat{\mu}(\xi)| \lesssim (1+|\xi|)^{-\beta/2}, \, \forall \xi \in \RR^{n}, 
\end{equation}
 for some $0<\alpha, \beta <n$. Then for any $q\geq (4n-4\alpha+2\beta)/\beta$, we have
\begin{equation}\label{eq:C0}
\V \widehat{f\mu}\V_{L^q(\RR^{n})}\lesssim_{q, \alpha, \beta}\V f\V_{L^{2}(\mu)}, \, \forall f\in L^{2}(\mu).
\end{equation}
\end{theorem}

One may ask whether this is still true for some $q<q_{n, \alpha, \beta}$. This is unknown for general $0<\alpha, \beta<n$. For $0<\alpha=\beta<1$, Hambrook and  Ł$\L$aba \cite{HambrookLaba} proved that the range $q\geq q_{1, \alpha, \alpha}$ is sharp.  By modifying the construction of Hambrook and $\L$aba \cite{HambrookLaba},  Chen \cite{ChenX} proved that the range 
is sharp for any $0<\alpha\leq \beta <1$. For $n\geq 2$, Hambrook and $\L$aba \cite{HambrookLabaB} proved that the range of $q$  is sharp when $n-1<\beta\leq \alpha<n$. We summarise the known results in the following.

%
%One question still unanswered is whether  the range $q\geq q_{n, \alpha, \beta}$ is sharp for general $0<\alpha, \beta<n$. 

\begin{theorem}
For any $n-1<\beta\leq \alpha<n$ there exists a Borel probability measure $\mu$ on $\RR$ with compact support such that  $\mu$  satisfies   \eqref{eq:A0} and \eqref{eq:B0}, but the  estimate \eqref{eq:C0} fails for any $q<q_{n,\alpha, \beta}$.  
\end{theorem}

Now we turn to the finite field setting.  The finite field version of restriction problem has already received a fair amount of attention, see  \cite{IosevichKohp}, \cite{IosevichKohs}, \cite{Lewko}, \cite{LewkoLewko}.

Let $\FF$ be a finite field, and $\FF^{n}$ be the $n$-dimensional vector space over this field. In this paper we will restrict our interest on prime fields only, which means that $|\FF|$  is always some (large) prime number.

We call $\mu$  a  measure on $\FF^{n}$ if $\mu(x) \geq 0$  for all $x\in \FF^{n}$. Let $f : \FF^{n}\longrightarrow \mathbb{C}$ be a complex value function.   The Fourier transform of $f\mu$ is defined as 
\begin{equation}\label{eq:ddd}
\widehat{f \mu}(\xi)=\sum_{x}e(-x\cdot \xi) f(x)\mu(x)
\end{equation}
where $e(-x \cdot \xi)=e^{-\frac{2\pi i x\cdot \xi}{|\FF|}}$ and the dot product 
\[
 x\cdot\xi =x_1\xi_1+\cdots +x_n\xi_n\,(\text{mod}\, |\FF|).
\] 
Here and in what follows $\sum_{z}$ means $\sum_{z\in \FF^{n}}$. Furthermore we define 
\[
\V f \V_{L^{q} (\FF^{n})}=\left(\sum_{x} |f(x)|^{q} \right)^{\frac{1}{q}} \text{ and } \V f \V_{L^{q} (\mu)}=\left(\sum_{x} |f(x)|^{q} \mu(x) \right)^{\frac{1}{q}}.
\]
For any $1\leq p, q \leq \infty$ let $R^{*}(p\rightarrow q)=R^{*}_{\mu}(p\rightarrow q)$ be the minimal constant such that 
\begin{equation}\label{eq:dd}
\V \widehat{f \mu} \V_{L^{q} (\FF^{n})}\leq R^{*}(p\rightarrow q) \V f \V_{L^{p} (\mu)}  
\end{equation}
hold for all function $f$ on the support of $\mu$. 

The finite field analogue of restriction problem determinates the values $p, q$ such that  $R^{*}(p\rightarrow q)\lesssim 1$, where the constant does not depend on $|\FF|$. 

We note that in the former research the measure $\mu$ will be some normalized ``surface measure" with support $S$, where $S$ will often be an algebraic variety. For instance $S$ is finite field version of  paraboloid, or sphere.  We refer to  \cite{IosevichKohp}, \cite{IosevichKohs}, \cite{Lewko}, \cite{LewkoLewko} for more details.

By adapting the argument of Mockenhaupt and Tao \cite{MT}, Green \cite{Green},  we have the following result for general measures on $\FF^{n}$.  Recall that $q_{n, \alpha, \beta}=(4n-4\alpha+2 \beta)/ \beta$.

\begin{theorem}\label{thm:main}
Let $\mu$ be a probability measure on $\FF^{n}$ such that 
\begin{equation}\label{eq:A}
\mu(x)\lesssim |\FF|^{-\alpha}, \forall x\in \FF^{n}
\end{equation}
and 
\begin{equation}\label{eq:B}
| \widehat{\mu}(\xi)| \lesssim |\FF|^{-\beta/2}, \forall \xi\neq 0
\end{equation}
for some $0<\alpha, \beta <n$. Then $R^{*}(2\rightarrow q)\lesssim 1$ for any $q\geq q_{n, \alpha, \beta}$.
\end{theorem}

In the following, we show that the range $q\geq q_{n, \alpha, \beta}$ in Theorem \ref{thm:main} is sharp for any $0<\alpha \leq \beta<n$. We construct the measure by ``combing" a random set and an arithmetic progression, which was inspired by the construction  of  Hambrook and  Ł$\L$aba \cite{HambrookLaba}, Chen \cite{ChenX}. 

 It is worth pointing out that it is unknown whether  the range $q\geq q_{n, \alpha, \beta}$ is sharp for general $0<\alpha\leq \beta <n$ for Euclidean setting.

\begin{theorem}\label{thm:sharpness}
Let $0<\beta<\alpha<n$.  Then  there exists a probability measure $\mu$ on $\FF^{n}$  such that $\mu$ satisfies   \eqref{eq:A}, \eqref{eq:B}, and if $q<q_{n, \alpha, \beta}$ then there is $\tau=\tau_{n, \alpha, \beta, q}>0$ such that  $R^{*}(2\rightarrow q)\gtrsim |\FF|^{\tau}$.
\end{theorem}

For the remaining case $0<\alpha=\beta<n$ we have the following result.

\begin{theorem}\label{thm:sharpnessalpha}
Let $0<\alpha<n$.  Then for any  $0<q< q_{n,\alpha, \alpha}$ there exists a probability measure $\mu$ on $\FF^{n}$  such that $\mu$ satisfies   \eqref{eq:A}, \eqref{eq:B}, and  $R^{*}(2\rightarrow q)\gtrsim |\FF|^{\tau}$ for some $\tau=\tau_{n, \alpha, q}>0$. 
\end{theorem}

We may improve Theorem \ref{thm:sharpnessalpha} by removing the dependence of exponential $q$, but we will not develop this point here.

\section{Preliminary}

In this section we show some basic properties of Fourier transform on $\FF^{n}$ for later use. We refer to Green \cite{Green}, Stein and Shakarchi \cite{Stein} for more details on discrete Fourier analysis.

The inverse Fourier transform of $f$ at $\xi$ is defined as 
\[
f^{\vee}(\xi)=\sum_{x} f(x)e(x\cdot \xi).
\]
The convolution  of function $f$ and $g$ is defined as 
\[
f \ast g(x)=\sum_{y}f(x-y)g(y).
\]
We collect some basic properties of Fourier transform on $\FF^{n}$. For more details see  \cite[Proposition 6.1]{Green}, \cite{Stein}. 

\begin{lemma} \label{lem:Fourier}

Let $f, g: \FF^{n}\rightarrow \mathbb{C}$ be functions. Then we have 

(i) (Plancherel) $\sum_{\xi} |\widehat{f}(\xi)|^{2}=|\FF|^{n}\sum_{x} |f(x)|^{2}$.

(ii) (Inversion) $(\widehat{f})^{\vee}(x)=(f^{\vee})^{\wedge}(x)=|\FF|^{n}f(x)$.

(iii) (Convolution)  $\widehat{ f \ast g}(\xi)= \widehat{f}(\xi) \widehat{g}(\xi)$.

(iv) (Symmetry) $\sum_{\xi} \widehat{f}(\xi)g(\xi)=\sum_{x}f(x)\widehat{g}(x)$.
\end{lemma}

By duality, we may also define $R^{*}(p\rightarrow q)$ as the minimal constant such that  
\begin{equation}\label{eq:restriction}
\Vert \widehat{f}\Vert_{L^{p'}(\mu)} \leq R^{*}(p\rightarrow q)\Vert f\Vert_{L^{q'}(\FF^{n})}
\end{equation}
for any function $f$ on $\FF^{n}$.
Here $q'$  denote the dual exponent of $q$ which means that $\frac{1}{q}+\frac{1}{q'}=1$.  We refer to \cite[Proposition 6.4]{Green} for a proof.

We note that the estimate \eqref{eq:restriction} is often called restriction estimate, while the estimate \eqref{eq:dd} is often called extension estimate. The terminology \emph{restriction problem} comes from the estimate \eqref{eq:restriction}.

\section{Proof of Theorem \ref{thm:main}}

Let $\mu$ be a probability  measure on $\FF^{n}$ satisfying the assumption of Theorem \ref{thm:main} and $f: \FF^{n}\rightarrow \mathbb{C}$ be a function. We will use the dual form of \eqref{eq:dd} to prove Theorem \ref{thm:main}. Therefore it is  sufficient to prove that for any $q\geq q_{n,\alpha, \beta}$,
\[
\V \widehat{f}\V_{L^{2}(\mu)}\lesssim \V f\V_{L^{q'}(\FF^{n})}
\]
for any function $f$ on $\FF^{n}$. Applying the basic properties of Fourier transform, Lemma \ref{lem:Fourier}, we obtain
\begin{equation*}
\begin{aligned}
\sum_{\xi } |\widehat{f}(\xi)|^{2}\mu(\xi)&= \sum_{\xi} \widehat{f}(\xi)\overline{\widehat{f}(\xi)\mu(\xi)}\\
&=\sum_{x}f(x) \overline{f \ast \mu ^{\vee}(x)} \\
&\leq \V f \V_{L^{q^{'}}(\FF^{n})} \V f \ast \mu ^{\vee}\V_{L^{q}(\FF^{n})}.
\end{aligned}
\end{equation*}
Let $K=\mu^{\vee}-\delta_{0}$. Since $f \ast \delta_{0}=f$ and $q\geq q_{n, \alpha, \beta}>2$, we have $q'<q$ and 
\[
\V f \V_{L^{q}(\FF^{n})}\leq \V f \V_{L^{q'}(\FF^{n})}.
\]
Thus it is sufficient to prove that for any $q\geq q_{n,\alpha, \beta}$,
\begin{equation*}
\V f \ast K\V_{L^q(\FF^{n})}\lesssim \V f \V_{L^{q'}(\FF^{n})}.
\end{equation*}
Since $|K(x)|\lesssim \FF^{-\beta/2}$ for any $x\in \FF^{n}$, we obtain 
\begin{equation}\label{eq:1toinfty}
\V f \ast K\V_{L^{\infty}(\FF^{n})}\leq \V K \V_{L^{\infty}(\FF^{n})}\V f\V_{L^{1}(\FF^{n})}\lesssim \FF^{-\beta/2} \V f\V_{L^{1}(\FF^{n})}.
\end{equation}
The inversion formula and the condition $\mu(x)\lesssim \FF^{-\alpha}$ implies that 
\[
|\widehat{K}(\xi)|=|(\mu^{\vee}-\delta_{0})^{\wedge}(\xi)|=|\FF|^{n}\mu(\xi)-1\lesssim |\FF|^{n-\alpha}.
\]
Combining with Plancherel identity we obtain 
\begin{equation}\label{eq:2to2}
\V f \ast K \Vert_{L^{2}(\FF^{n})}^{2}=\FF^{-n}\sum_{\xi}|\widehat{f}(\xi)|^{2}|\widehat{K}(\xi)|^{2}\lesssim |\FF|^{2n-2\alpha}\sum_{x} |f(x)|^{2}.
\end{equation}
Applying \eqref{eq:1toinfty}, \eqref{eq:2to2} and Riesz-Thorin interpolation theorem we obtain that for any $q\geq q_{n,\alpha,\beta}$
\begin{equation*}
\V f\ast K\V_{L^{q}(\FF^{n})}\lesssim \V f\V_{L^{q'}(\FF^{n})}.
\end{equation*}
Note that the constant depends on $\mu$ only, which finishes the proof.

\begin{remark}
The above arguments also works for  general finite fields. The only difference is that we take a non-trivial character instead of $e(-x\cdot \xi)$ in  the definition of $\widehat{f \mu}$ at \eqref{eq:ddd}.
\end{remark}

\section{Construction of measures}

Let $0<\alpha\leq \beta<n$ and $N$ be an integer such that  $N^{n}\approx |\FF|^{\alpha-\beta/2}$.
Let $A=\{1, \cdots, N\}^{n}$. Define 
\begin{equation}\label{eq:Bohr}
A^{*}=\{\xi \in \FF^{n}: |\widehat{A}(\xi)|\geq |A|/10\}.
\end{equation}
Note that $A^{*}$ is often called Bohr neighbourhood of $A$. We have the following lower bound for the  cardinality of $A^{*}$. 

\begin{lemma}\label{lem:A}
Use the above notation. We have
\begin{equation}
|A^{*}|\gtrsim_{n} \left(\frac{|\FF|}{N}\right)^{n}\gtrsim |\FF|^{n-\alpha+\beta/2}.
\end{equation}
\end{lemma}
\begin{proof}
Let  $\xi=(\xi_{1},\cdots, \xi_{n})$ with $|\xi_{j}|\leq |\FF|/Nn10$ for all $1\leq j\leq n$, then
\[
|x\cdot \xi|\leq |\FF|/10, \, \forall x\in A.
\]
It follows that
\[
|\sum_{x\in A}e(x\cdot \xi)|\geq |\sum_{x\in A}Re(e(x\cdot \xi))|\geq |A|/10.
\]
Thus we conclude that $\xi\in A^{*}$  when $|\xi_{j}|\leq |\FF|/Nn10$ for all $1\leq j\leq n$, which completes the proof.
\end{proof}

By a random construction, Hayes \cite{Hayes} (see an exposition of Babai \cite[Theorem 5.14]{Babai} obtained the following result.  The author \cite{ChenS} reproved Hayes's result by studying a different random model. We state it in the following form for our later use, and we show an outline of the argument in \cite{ChenS} for completeness.  For $B\subset \FF^{n}$ we denote by ${\bf 1}_{B}$ the characteristic function of $B$.
\begin{theorem}\label{thm:H}
Let $0<\alpha<n$ there exists a subset $E\subset \FF^{n}$ such that  $|E|\approx |\FF|^{\alpha}$ and 
\begin{equation}
|\widehat{{\bf 1}_E}(\xi)|\lesssim_{n} \sqrt{|E|} \sqrt{\log |\FF|}, \, \forall \xi \neq 0.
\end{equation}
Furthermore for any (small) $\eta>0$,
\begin{equation}
|E \cap A| \leq \eta |A|
\end{equation}
provided $\FF$ is large enough.
\end{theorem}
\begin{proof}[Proof sketch] 
Let $\delta=|\FF|^{\alpha- n}$. Then we choose each point of $\FF^{n}$ with probability $\delta$ and remove it with probability $1-\delta$, all choices being independent of each other. Let $E=E^{\omega}$ be the collection of these chosen points. 

First of all since the cardinality of $E$ forms a binomial distribution with total number $\FF^{n}$ and the succeed probability $\delta$, we have $|E|\approx |\FF|^{\alpha}$ with high probability  when $|\FF|$ is large enough. Secondly observe that for $\xi \in \FF^{n}$, 
\[
\widehat{{\bf 1}_{E}}(\xi)=\sum_{x\in \FF^{n}} e(-x\cdot\xi) {\bf 1}_{E}(x),
\]
which is the sum of a sequence of independent random variables. By applying  large deviations estimate we obtain that  
\[
|\widehat{{\bf 1}_{E}}(\xi)|\lesssim_{n} \sqrt{|E|} \sqrt{\log |\FF|}, \, \forall \xi \neq 0, 
\]
with high probability. For the subset $A$, the size of $A\cap E$ forms a binomial distribution, and hence $|A\cap E|\approx |A|\delta$ with high probability. Thus we conclude that with high probability  $E$ satisfies our need, which finish the proof.
\end{proof}

By ``combining" the set $A$ of Lemma \ref{lem:A} and $E$ of Theorem \ref{thm:H},  we obtain  the desired measure. We show the construction of the desired measure and its basic properties in the following.

\begin{theorem}\label{thm:construction}
 For any $0<\beta\leq \alpha<n$, let $A=\{1,\cdots, N\}^{n}$ where $N^{n}\approx |\FF|^{\alpha-
 \beta/2}$. Let $E$ be a set such that satisfies the properties in Theorem \ref{thm:H} for $\alpha$. Define 
\[
 \mu = \frac{{\bf 1}_{E} +{\bf 1}_{A} }{|E|+|A|}.
\]
Then $\mu$ is a probability measure on $\FF^{n}$, 
\begin{equation}\label{eq:regular}
\mu(x)\lesssim |\FF|^{-\alpha}, \, \forall x \in \FF^{n}, 
\end{equation}
and
\begin{equation}\label{eq:Salem}
|\widehat{\mu}(\xi)|\lesssim |\FF|^{-\alpha/2} \sqrt{\log |\FF|}+|\FF|^{-\beta/2}, \, \forall \xi \neq 0.
\end{equation}
Furthermore  let $f={\bf 1}_{A}$, then 
\begin{equation}\label{eq:largeFoureir}
| \widehat{ f \mu} (\xi)| \gtrsim |\FF|^{-\beta/2}, \, \forall \xi \in A^{*}.
\end{equation}
The definition of $A^{*}$ is at \eqref{eq:Bohr}.
\end{theorem}
\begin{proof}
Since $\sum_{x}\mu(x)=1$ and 
$|E|+|A|\approx |E|\approx |\FF|^{\alpha}$,  we obtain that $\mu$ is a probability measure on $\FF^{n}$ and $\mu(x)\lesssim |\FF|^{-\alpha}$ for all $x\in \FF^{n}$. 

By Theorem \ref{thm:H}, we have that 
\[
|\widehat{ {\bf 1}_{E}}(\xi)|\lesssim_{n} |\FF|^{\alpha/2}\sqrt{\log |\FF|}, \, \forall \xi \neq 0.
\]
By the definition of Fourier transform, $|\widehat{ {\bf 1}_A}(\xi)|\leq |A|$. It follows that the estimate \eqref{eq:Salem} holds.

For $f={\bf 1}_{A}$ we have 
\[
f\mu=\frac{{\bf 1}_{E\cap A}+ {\bf 1}_{A}}{|E|+|A|}.
\]
Let $\eta=1/100$ in Theorem \ref{thm:H}, then $|A\cap E|\leq |A|/100$ when $|\FF|$ is large enough. Thus for $\xi \in A^{*}$ we have 
\[
|\widehat{\bf 1_{E\cap A}}(\xi)+ \widehat{{\bf 1}_{A}}(\xi)|\geq |A|/10-|A\cap E|\gtrsim |A|,
\]
and then we arrive the estimate \eqref{eq:largeFoureir}. This completes  the proof.
\end{proof}

\section{Proof of Theorem \ref{thm:sharpness} and Theorem \ref{thm:sharpnessalpha}}

\begin{proof}[Proof of Theorem \ref{thm:sharpness}]
Let $0<\beta<\alpha<n$ and  $\mu$ be the measure of Theorem \ref{thm:construction}.  We use notation from Theorem \ref{thm:construction} in the following. Since $\beta<\alpha$ the estimate \eqref{eq:Salem} becomes 
\begin{equation*}
|\widehat{\mu}(\xi)|\lesssim |\FF|^{-\beta/2}, \, \forall \xi \neq 0.
\end{equation*}
Thus the measure $\mu$ satisfies the conditions \eqref{eq:A}, \eqref{eq:B}  of Theorem \ref{thm:sharpness}.  

Let $f={\bf 1}_{A}$, then the estimate \eqref{eq:largeFoureir} implies that 
\begin{equation*}
\V \widehat{f\mu} \V_{L^{q}(\FF^{n})}\gtrsim |\FF|^{-\beta/2}|A^{*}|^{1/q}\gtrsim |\FF|^{\frac{2n-2\alpha+\beta-q\beta}{2q}}.
\end{equation*}
Meanwhile,  we have 
\begin{equation}
\V f \V_{L^{2}(\mu)}\lesssim |\FF|^{-\beta/4}. 
\end{equation}
It follows that 
\begin{equation}\label{eq:R}
R^{*}(2\rightarrow q)\geq \frac{\V \widehat{f\mu }\V_{L^{q}(\FF^{n})}}{\V f \V_{L^{2}(\mu)}}\gtrsim |\FF|^{\frac{4n-4\alpha+2\beta-q\beta}{4q}},
\end{equation}
which is our claim.
\end{proof}

\begin{proof}[Proof of Theorem \ref{thm:sharpnessalpha}]
Let $0<\alpha<n$ and $q<\frac{4n-2\alpha}{\alpha}$. Recall that $q_{n, \alpha, \beta}=\frac{4n-4\alpha+2\beta}{\beta}$. 
There is a (small) $\varepsilon>0$ such that 
\[
q<q_{n, \alpha', \beta'}
\]
where $\alpha'=\alpha+\varepsilon$ and $\beta'=\alpha$. Applying the same argument  as in the proof of Theorem \ref{thm:sharpness} to $\alpha'$ and $ \beta'$, we obtain that  there is a measure $\mu$ on $\FF^{n}$ such that 
\[
\mu(x)\lesssim |\FF|^{-\alpha'}\lesssim |\FF|^{-\alpha}, \forall x\in \FF^{n},
\]
and 
\[
| \widehat{\mu}(\xi)| \lesssim |\FF|^{-\beta'/2}\lesssim |\FF|^{-\alpha/2}, \forall \xi\neq 0.
\] 
Furthermore, there is a function $f$ such that 
\begin{equation*}
R^{*}(2\rightarrow q)\geq \frac{\V \widehat{f\mu }\V_{L^{q}(\FF^{n})}}{\V f \V_{L^{2}(\mu)}}\gtrsim |\FF|^{\frac{4n-4\alpha'+2\beta'-q\beta'}{4q}}.
\end{equation*}
Since $q<q_{n, \alpha', \beta'}$, we complete the proof.
\end{proof}

\section{Remarks}

1.  Our construction also the implies the following necessary condition of $R^{*}(p\rightarrow q)\lesssim 1$ for $1\leq p, q\leq \infty$.

\begin{theorem}
Let $0<\beta \leq \alpha<n$, and $\mu$ be a probability measure on $\FF^{n}$ which satisfies the conditions \eqref{eq:A}, \eqref{eq:B}. For $1\leq p, q\leq \infty$ if $R^{*}(p, q)\lesssim 1$, then 

\begin{equation}
q \geq 
  \begin{cases}
   \infty  & p=1, \\
   \frac{2n-2\alpha+\beta}{\beta} & p=\infty,\\
   \frac{p(2n-2\alpha+\beta)}{(p-1)\beta} &  1<p<\infty.
  \end{cases}
\end{equation}
\end{theorem}
\begin{proof}[Proof sketch] 
Assume first $1\leq p<\infty$.  We use the same argument and notation as in the proof of Theorem \ref{thm:sharpness}.  For $1\leq p<\infty$ the estimate \eqref{eq:R} becomes 
\[
R^{*}(p\rightarrow q)\gtrsim |\FF|^{\frac{2n-2\alpha+\beta-q\beta}{2q}+\frac{\beta}{2 p}},
\]
which finishes the proof for   $1\leq p<\infty$. The same proof works for  $q=\infty$.
\end{proof}

It is clear that $R^{*}(1\rightarrow \infty)\leq1$.  For  $1<p\leq \infty$, we do not know whether the necessary condition is also sufficient.

2.  In the following we show finite field analogues of remarks in Mitsis \cite{Mitsis}. It is easy to see that if $\mu$ is a probability measure on $\FF^{n}$ such that $\mu(x)\lesssim |\FF|^{-\alpha}, \, \forall x\in \FF^{n},$ then $|\text{Supp} (\mu)|\gtrsim |\FF|^{\alpha}$. Furthermore if $\mu(x)\approx |\FF|^{-\alpha}$ for all $x\in \text{Supp}(\mu )$ then $|\text{Supp}(\mu)|\approx |\FF|^{\alpha}$, and we call the measure $\mu$ a $\alpha$-Ahlfors-David regular measure. 

By applying the basic properties of Fourier transform,  Lemma \ref{lem:Fourier}, we obtain the following easy facts. 
\begin{proposition} Let $\mu$ be a probability measure on $\FF^{n}$ and $|\widehat{\mu}(\xi)|\lesssim |\FF|^{-\beta/2}$ for any $\xi\neq 0$. Then we have $
|\text{Supp}(\mu)|\gtrsim |\FF|^{\beta}$,
and $
\mu(x)\lesssim |\FF|^{-\beta /2}$.
\end{proposition}
\begin{proof}
Applying Cauchy-Schwarz inequality and Plancherel identity, we obtain
\begin{equation}
\begin{aligned}
1=\left(\sum_{x}\mu(x)\right)^{2}&\leq |\text{Supp}(\mu)| \sum_{x}\mu(x)^{2}\\
&=|\text{Supp}(\mu)| |\FF|^{-n}\sum_{\xi}|\widehat{\mu(\xi)}|^{2}.
\end{aligned}
\end{equation}
Applying the condition $|\widehat{\mu}(\xi)|\lesssim |\FF|^{-\beta/2}$, we arrive the estimate $|\text{Supp}(\mu)|\gtrsim |\FF|^{\beta}$.

The inversion formula of  $\mu$ is
\[
\mu(x)=|\FF|^{-n}\sum_{\xi}\widehat{\mu}(\xi)e(-x\cdot \xi).
\] 
Thus we obtain the estimate $\mu(x)\lesssim |\FF|^{-\beta/2}.$
\end{proof}

Applying the above proposition and Theorem \ref{thm:main}, we obtain the following restriction estimate if we only know the Fourier coefficients of $\mu$. We emphasize that the following corollary is the finite field version of \cite[Corollary 3.1]{Mitsis}.

\begin{corollary}
Let $\mu$ be a probability measure on $\FF^{n}$ and $|\widehat{\mu}(\xi)|\lesssim |\FF|^{-\beta/2}$ for any $\xi\neq 0$. Then for any $q\geq \frac{4n}{\beta}$ we have $R^{*}(2\rightarrow q)\lesssim 1$.
\end{corollary}

3. The following problem is a finite field version of a problem in Mitsis \cite{Mitsis},  Mattila \cite[Problem 20]{Mattila2004}, Mattila \cite[Chapter 3.6]{Mattila}. The following part is closely relate to Salem sets in Euclidean space, and finite field version of Salem sets. We refer to \cite[Chapter 3.6]{Mattila} for more details on Salem sets in Euclidean spaces. Iosevich and Rudnev \cite{IosevichRudnev} introduced finite field version of Salem sets, and they showed some examples in their paper. See also \cite{ChenS} for a random  construction of  Salem sets. 

\begin{question}\label{que:AD}
For any  $0<\alpha<n$, does there exist a probability  measure $\mu$ on $\FF^{n}$ such that $\mu(x)\approx |\FF|^{-\alpha}$ for all $x\in \text{Supp}(\mu)$ and $|\widehat{\mu}(\xi)|\lesssim |\FF|^{-\alpha/2}$ for all $\xi\neq 0$ ?
\end{question}

We may also have the following version for subset of $\FF^{n}$.  

\begin{question}\label{que:lll}
Let $0<\alpha<n$. Does there exist a subset $E \subset \FF^{n}$ such that $|E|\approx |\FF|^{\alpha}$ and $|\widehat{ {\bf 1}_{E}}(\xi)|\lesssim \sqrt{|E|}$ for all $\xi\neq 0$ ?
\end{question}

Recall that Theorem \ref{thm:H} claims that for any $0<\alpha<n$ there exists a subset $E\subset \FF^{n}$ such that $|E|\approx |\FF|^{\alpha}$ and 
\[
|\widehat{ {\bf 1}_{E}}(\xi)|\lesssim \sqrt{|E|} \sqrt{\log |\FF|}, \, \forall \xi \neq 0.
\] 
Thus Question \ref{que:lll} ask that can we remove the factor $\sqrt{\log |\FF|}$. We note that the above constants may depends on $n, \alpha$, but does not depend on $|\FF|$.

We note that  the answer of above questions are positive when $\alpha=n-1, n\geq 2$. Iosevich and Rudnev \cite{IosevichRudnev} proved that discrete paraboloid and discrete sphere satisfy the need for Question \ref{que:lll}, and hence Question \ref{que:AD}, for $\alpha=n-1$. It seems that $\alpha=n-1$ is the only known value such that the above questions have positive answer.

\end{document}